\documentclass[12pt]{article}

\usepackage[dvipdfm, bookmarksopen=true]{hyperref}
\usepackage[centertags]{amsmath}
\usepackage{pstricks}
\usepackage{amsfonts}
\usepackage{amssymb}
\usepackage{amsthm}
\usepackage{cite}

\theoremstyle{plain}
\newtheorem{thm}{Theorem}[section]
\newtheorem{cor}[thm]{Corollary}
\newtheorem{lem}[thm]{Lemma}

\newtheorem{prop}[thm]{Proposition}

\theoremstyle{definition}

\newtheorem{defn}[thm]{Definition}

\theoremstyle{remark}
\newcommand{\beast}{\begin{eqnarray*}}
\newcommand{\eeast}{\end{eqnarray*}}

\title{ \bf Lit-only sigma-game on nondegenerate graphs
}

\author{Hau-wen Huang}

\date{}

\begin{document}
\maketitle

\begin{abstract}
A configuration of the lit-only $\sigma$-game on a graph $\Gamma$ is an assignment of one of two states, {\it on} or {\it off}, to each vertex of $\Gamma.$ Given a configuration, a move of the lit-only $\sigma$-game on $\Gamma$ allows the player to choose an {\it on} vertex $s$ of $\Gamma$ and change the states of all neighbors of $s.$ Given an integer $k$, the underlying graph $\Gamma$ is said to be $k$-lit if for any configuration, the number of {\it on} vertices can be reduced to at most $k$ by a finite sequence of moves. We give a description of the orbits of the lit-only $\sigma$-game on nondegenerate graphs $\Gamma$ which are not line graphs. We show that these graphs $\Gamma$ are $2$-lit and provide a linear algebraic criterion for $\Gamma$ to be $1$-lit.
\end{abstract}


{\footnotesize {\bf 2010  MSC Primary:} 05C57; {\bf Secondary:} 15A63, 20F55}

\section{Introduction}\label{Introduction}

The notion of the {\it $\sigma$-game} on finite directed graphs $\Gamma$ without multiple edges was first introduced by Sutner \cite{ks:89} in 1989. A {\it configuration} of the $\sigma$-game on $\Gamma$ is an assignment of one of two states, {\it on} or {\it off}, to each vertex of $\Gamma$. Given a configuration, a {\it move} consists of choosing a vertex of $\Gamma$, followed by changing the states of all of its neighbors.  
If only {\it on} vertex is allowed to choose in each move, we come to the variation: {\it lit-only $\sigma$-game}. Starting from an initial configuration, the goal of the lit-only $\sigma$-game on $\Gamma$ is to minimize the number of {\it on} vertices of $\Gamma$, or to reach an assigned configuration by a finite sequence of moves.

Given an integer $k$, the underlying graph $\Gamma$ is said to be {\it $k$-lit} if for any configuration, the number of {\it on} vertices can be reduced to at most $k$ by a finite sequence of the moves.
More precisely, we are interested in the orbits of the lit-only $\sigma$-game on $\Gamma$ and the smallest integer $k$, the {\it minimum light number} of $\Gamma$ \cite{xwykw:07}, for which $\Gamma$ is $k$-lit. The notion of lit-only $\sigma$-games occurred implicitly in the study of equivalence classes of Vogan diagrams. The Borel-de Siebenthal theorem \cite{abjs:49} showed that every Vogan diagram is equivalent to one with a single painted vertex, which implies that each simply-laced Dynkin diagram is $1$-lit. The equivalence classes of Vogan diagrams were described by Chuah and Hu \cite{mkc:04}. A conjecture made by Chang \cite{gchang:05,gchang2:05} that any tree with $k$ leaves is $\lceil k/2 \rceil$-lit was confirmed by Wang and Wu \cite{xwykw:07}, where the name ``lit-only $\sigma$-game" was coined.

The lit-only $\sigma$-game on a simple graph $\Gamma$ is simply the natural action of a certain subgroup $H_{\it \Gamma}$ of the general linear group over $\mathbb{F}_2$ \cite{xwykw:07}. Under the assumption that $\Gamma$ is the line graph of a simple graph $G,$ Wu \cite{ykw:08} described the orbits of the lit-only $\sigma$-game on $\Gamma$ and gave a characterization for the minimum light number of $\Gamma$. Moreover, if $G$ is a tree of order $n\geq 3 $, Wu showed that $H_{\it \Gamma}$ is isomorphic to the symmetric group on $n$  letters. Weng and the author \cite{hw:09-2} determined the structure of $H_{\it \Gamma}$ without any assumption on $G$. The lit-only $\sigma$-game on a simple graph $\Gamma$ can also be considered as a representation $\kappa_{\it \Gamma}$ of the simply-laced Coxeter group $W_{\it \Gamma}$ over $\mathbb{F}_2$ \cite{hw:08-1}. The dual representation of $\kappa_{\it \Gamma}$  preserves a certain symplectic form $B_{\it \Gamma}$. The two representations are equivalent whenever the form $B_{\it \Gamma}$ is nondegenerate. From this viewpoint it is natural to partition simple connected graphs into two classes according as $B_{\it \Gamma}$ is degenerate or nondegenerate.

In this paper we treat nondegenerate graphs $\Gamma$ which are not line graphs. We show that $H_{\it \Gamma}$ is isomorphic to an orthogonal group, followed by a description of the orbits of lit-only $\sigma$-game on $\Gamma$ (Theorem \ref{thm1.2}). Moreover we show that these graphs $\Gamma$ are 2-lit and provide a linear algebraic criterion for $\Gamma$ to be $1$-lit (Theorem \ref{thm1.3}). Combining Theorem~\ref{thm1.2}, Theorem~\ref{thm1.3} and those in \cite{hw:09-2} and \cite{ykw:08}, the study of the lit-only $\sigma$-game on nondegenerate graphs is quite completed and the focus for further research is on degenerate graphs.

\section{Preliminaries}\label{Preliminaries}
From now on, let $\Gamma=(S,R)$ denote a finite simple connected graph with vertex set $S$ and edge set $R.$ Let $\mathbb{F}_2$ denote the two-element field $\{0,1\}.$ Let $V$ denote a $\mathbb{F}_2$-vector space that has a basis $\{\alpha_s~|~s\in S\}$ in one-to-one correspondence with $S.$ Let $V^\ast$ denote the dual space of $V.$ For each $s\in S$ we define $f_s\in V^\ast$ by
\begin{eqnarray}\label{e1.1}
f_s(\alpha_t)=\left\{
\begin{array}{ll}
1 \qquad &\hbox{if $s=t,$}\\
0 \qquad &\hbox{else}
\end{array}
\right.
\end{eqnarray}
for all $t\in S$. The set $\{f_s~|~s\in S\}$ forms a basis of $V^\ast$ and is called the {\it basis of $V^\ast$ dual to $\{\alpha_s~|~s\in S\}.$} Each configuration $f$ of the lit-only $\sigma$-game on $\Gamma$ is interpreted as the vector
\begin{align}\label{e1.4}
\sum_{\hbox{\scriptsize {\it on} vertices $s$}} f_s\in V^\ast,
\end{align}
if all vertices of $\Gamma$ are assigned the {\it off} state by $f,$ we interpret  (\ref{e1.4}) as the zero vector of $V^\ast.$ Given $s\in S$ and $f\in V^\ast$ observe that $f(\alpha_s)=1$ (resp. 0) if and only if the vertex $s$ is assigned the {\it on} (resp. {\it off} ) state by $f.$ For each $s\in S$ define a linear transformation $\kappa_s:V^\ast\rightarrow V^\ast$ by
\begin{eqnarray}\label{e1.3}
\kappa_sf = f+f(\alpha_s)\sum_{st\in R}f_{t} \qquad \quad \hbox{for all $f\in V^\ast.$}
\end{eqnarray}
Fix a vertex $s$ of $\Gamma.$ Given any $f\in V^\ast,$ if the state of $s$ is {\it on} then $\kappa_sf$ is obtained from $f$ by changing the states of all neighbors of $s,$ and $\kappa_sf=f$ otherwise. Therefore we may view $\kappa_s$ as the move of the lit-only $\sigma$-game on $\Gamma$ for which we choose the vertex $s$ and change the states of all neighbors of $s$ if the state of $s$ is {\it on}. In particular $\kappa_s^2=1$, the identity map on $V^\ast,$ and so $\kappa_s\in {\rm GL}(V^\ast),$ the general linear group of $V^\ast$. The subgroup $H=H_{\it \Gamma}$ of ${\rm GL}(V^\ast)$ generated by the $\kappa_s$ for all $s\in S$ was first mentioned by Wu \cite{xwykw:07}, which is called the {\it flipping group} of $\Gamma$ in \cite{hw:08-1} and the {\it lit-only group} of $\Gamma$ in \cite{ykw:08}.

The simply-laced Coxeter group $W=W_{\it \Gamma}$ associated with $\Gamma=(S,R)$ is defined by generators and relations. The generators are the elements of $S$ and the relations are
\begin{eqnarray*}
s^2&=&1,\\
(st)^2&=&1 \qquad \hbox{if $st\not\in R,$}\\
(st)^3&=&1 \qquad \hbox{if $st\in R$.}
\end{eqnarray*}
By \cite[Theorem~3.2]{hw:08-1}, there exists a unique representation $\kappa=\kappa_{\it \Gamma}:W\rightarrow {\rm GL}(V^\ast)$ such that  $\kappa(s)=\kappa_s$ for all $s\in S$. Clearly $\kappa(W)=H$. Given any $f,g\in V^\ast$ observe that $g$ can be obtained from $f$ by a finite sequence of the moves of the lit-only $\sigma$-game on $\Gamma$ if and only if there exists $w\in W$ such that $g=\kappa(w)f.$ Given an integer $k$, the underlying graph $\Gamma$ is $k$-lit if and only if for each $\kappa(W)$-orbit $O$ on $V^\ast$, there exists a subset $K$ of $S$ with size at most $k$ such that $\sum_{s\in K} f_s\in O$.



A symplectic form $B=B_{\it \Gamma}$ on $V$ is defined by
\begin{eqnarray}\label{e1.2}
B(\alpha_s,\alpha_t)=\left\{
\begin{array}{ll}
1 \qquad &\hbox{if $st\in R,$}\\
0 \qquad &\hbox{else}
\end{array}
\right.
\end{eqnarray}
for all $s,t\in S$ \cite{rd:05}. The {\it radical} of $V$ (relative to $B$) is the subspace of $V$ consisting of the vectors $\alpha$ that satisfy $B(\alpha,\beta)=0$ for all $\beta\in V.$ The form $B$ is said to be {\it degenerate} whenever the radical of $V$ is nonzero and {\it nondegenerate} otherwise. The graph $\Gamma$ is said to be {\it degenerate} whenever the form $B$ is degenerate, and {\it nondegenerate} otherwise. The form $B$ induces a linear map $\theta:V\rightarrow V^*$ given by
\begin{eqnarray}
\theta(\alpha)\beta=B(\alpha,\beta)  \qquad \qquad \hbox{for all $\alpha,\beta\in V.$}  \label{e1.5}
\end{eqnarray}
Since the kernel of $\theta$ is the radical of $V$ and the matrix representing $B$ with respect to the basis $\{\alpha_s~|~s\in S\}$ is the adjacency matrix of $\Gamma$ over $\mathbb{F}_2$, the following lemma is straightforward.

\begin{lem}\label{lem4.2}
Let $A$ denote the adjacency matrix of $\Gamma$ over $\mathbb{F}_2.$ Then the following are equivalent:
\begin{enumerate}
\item $\Gamma$ is a nondegenerate graph.

\item $\theta$ is an isomorphism of vector spaces.

\item $A$ is invertible.
\end{enumerate}
\end{lem}


The purpose of this paper is to investigate the lit-only $\sigma$-game on nondegenerate graphs which are not line graphs. It is natural to ask how to determine if a nondegenerate graph is a line graph. Here we give two characterizations of nondegenerate line graphs.

\begin{lem}\label{lem5.1}
Assume that $\Gamma$ is the line graph of a simple connected graph $G$ of order $n$. Then $\theta(V)$ has dimension $n-1$ if $n$ is odd and has dimension $n-2$ if $n$ is even.
\end{lem}
\begin{proof}
Let $U$ denote the vertex space of $G$ over $\mathbb{F}_2.$ Define a linear map $\mu:V\to U$ by
\begin{eqnarray*}
\mu(\alpha_s)=u+v \qquad \quad \hbox{for all $s\in S$},
\end{eqnarray*}
where $u$ and $v$ are the two endpoints of $s$ in $G.$ Since $G$ is connected,  the image of $\mu$ is the subspace of $U$ consisting of these vectors each of which equals the sum of an even number of vertices of $U.$ Define a linear map $\lambda:U\to V^\ast$ by
\begin{eqnarray*}
\lambda(u)\alpha_s = \left\{
\begin{array}{ll}
1 \quad &\hbox{if $u$ is incident to $s$ in $G$,}\\
0 \quad &\hbox{else}
\end{array}
\right.
\end{eqnarray*}
for all $u\in U$ and for all $s\in S$. There is only one nonzero vector, the sum of all vertices of $G$, in the kernel of $\lambda$. Since $\theta=\lambda\circ \mu$ and by the above comments, the result follows.
\end{proof}

A {\it claw} is a tree with one internal vertex and three leaves.
A simple graph is said to be {\it claw-free} if it does not contain a claw as an induced subgraph. A {\it cut-vertex} of $\Gamma$ is a vertex of $\Gamma$ whose deletion increase the number of components. A {\it block} of $\Gamma$ is a maximal connected subgraph of $\Gamma$ without cut-vertices. A {\it block graph} is a simple connected graph in which every block is a complete graph.

\begin{lem}\label{lem5.2}
{\rm \cite[Theorem 8.5]{ha:72}.}
Let $\Gamma$ denote a simple connected graph. Then $\Gamma$ is the line graph of a tree if and only if $\Gamma$ is a claw-free block graph.
\end{lem}

The following proposition follows by combining Lemmas~\ref{lem4.2}--\ref{lem5.2}.

\begin{prop}\label{lem5.4}
Let $\Gamma$ denote a simple connected graph. Then the following are equivalent:
\begin{enumerate}
\item $\Gamma$ is a nondegenerate line graph.

\item $\Gamma$ is the line graph of an odd-order tree.

\item $\Gamma$ is a claw-free block graph of even order.
\end{enumerate}
\end{prop}



\section{Main results}


A {\it quadratic form $Q$ on $V$ {\rm (}associated with $B${\rm )}} is a function $Q:V\to \mathbb{F}_2$ satisfying
\begin{gather}\label{e2.2}
Q(\alpha+\beta)=Q(\alpha)+Q(\beta)+B(\alpha,\beta) \qquad \quad \hbox{for all $\alpha,\beta\in V.$}
\end{gather}
Let ${\rm GL}(V)$ denote the general linear group of $V.$ Given a quadratic form $Q$ on $V$, the {\it orthogonal group} with respect to $Q$ is the subgroup of ${\rm GL}(V)$ consisting of all $\sigma\in{\rm GL}(V)$ such that $Q(\sigma\alpha)=Q(\alpha)$ for all $\alpha\in V.$ Given a basis $P$ of $V$ we define $Q_P$ to be the quadratic form on $V$ satisfying $Q_P(\alpha)=1$ for all $\alpha\in P$.

For the rest of this paper, the form $B$ is assumed to be nondegenerate. Moreover let $Q=Q_P$ where $P=\{\alpha_s~|~s\in S\}$ and let $O(V)$ denote the orthogonal group with respect to $Q$. We now state the main results of this paper, which are Theorem~\ref{thm1.2}, Theorem~\ref{thm1.3}, and Corollary~\ref{cor1.1}.


\begin{thm}\label{thm1.2}
Assume that $\Gamma$ is a nondegenerate graph but not a line graph. Then $\kappa(W)$ is isomorphic to $O(V)$. Moreover the $\kappa(W)$-orbits on $V^\ast$ are
\begin{gather*}
\{0\}, \qquad \theta (Q^{-1}(0)\setminus\{0\}), \qquad \theta(Q^{-1}(1)).
\end{gather*}
\end{thm}

Under the assumption that $B$ is nondegenerate, the number $|S|=2m$ is even and there exists a basis $\{\beta_1,\gamma_1,\ldots,\beta_m,\gamma_m\}$ of $V$ such that $B(\beta_i,\beta_j)=0$, $B(\gamma_i,\gamma_j)=0$ and
$$
B(\beta_i,\gamma_j)=
\left\{
\begin{array}{ll}
1 \qquad &\hbox{if $i=j$,}\\
0 \qquad &\hbox{else}
\end{array}
\right.
$$
for all $1\leq i,j\leq m$.
Such a basis $\{\beta_1,\gamma_1,\ldots,\beta_m,\gamma_m\}$ of $V$ is called a {\it symplectic basis} of $V$. The {\it Arf invariant} of $Q$ is defined to be
$$
{\rm Arf}(Q)=\sum_{i=1}^m Q(\beta_i) Q(\gamma_i),
$$
which is independent on the choice of the symplectic basis $\{\beta_1,\gamma_1,\ldots,\beta_m,\gamma_m\}$ of $V$ (for example see \cite{arf:41} or \cite[Theorem~13.13]{grove:02}). Any two quadratic forms over $\mathbb{F}_2$ are equivalent if and only if they have the same Arf invariant and the underlying spaces have the same dimension (for example see \cite{arf:41} or \cite[Proposition~13.14]{grove:02}). The order of $O(V)$ and the sizes of nontrivial $O(V)$-orbits on $V$ are given as follows (cf.  \cite[Chapter~14]{grove:02}). If ${\rm Arf}(Q)=0$ then
\begin{eqnarray*}
\big|O(V)\big|&=&
2^{m^2-m+1}(2^m-1)(2^2-1)(2^4-1)\cdots(2^{2m-2}-1),\\
\big|Q^{-1}(1)\big|&=&2^{2m-1}-2^{m-1},\\
\big|Q^{-1}(0)\setminus\{0\}\big|&=&2^{2m-1}+2^{m-1}-1.
\end{eqnarray*}
If ${\rm Arf}(Q)=1$ then
\begin{eqnarray*}
\big|O(V)\big|&=&
2^{m^2-m+1}(2^{m}+1)(2^2-1)(2^4-1)\cdots(2^{2m-2}-1),\\
\big|Q^{-1}(1)\big|&=&2^{2m-1}+2^{m-1},\\
\big|Q^{-1}(0)\setminus\{0\}\big|&=&2^{2m-1}-2^{m-1}-1.
\end{eqnarray*}





For each $s\in S$ there exists $\alpha_s^\vee\in V$ such that
\begin{eqnarray}\label{e10.5}
B(\alpha_s^\vee,\alpha_t)=\left\{
\begin{array}{ll}
1 \qquad &\hbox{if $s=t,$}\\
0 \qquad &\hbox{else}
\end{array}
\right.
\end{eqnarray}
for all $t\in S$. The set $\{\alpha_s^\vee~|~s\in S\}$ forms a basis of $V$ and is called  {\it the basis of $V$ dual to $\{\alpha_s~|~s\in S\}$ {\rm (}with respect to $B${\rm )}}.

\begin{thm}\label{thm1.3}
Assume that $\Gamma=(S,R)$ is a nondegenerate graph but not a line graph. Then $\Gamma$ is $2$-lit. Moreover the following are equivalent:
\begin{enumerate}
\item $\Gamma$ is $1$-lit.

\item The restriction of $Q$ to $\{\alpha_s^\vee~|~s\in S\}$ is surjective.
\end{enumerate}
\end{thm}

When the nondegenerate graph $\Gamma$ is bipartite, Theorem~\ref{thm1.3} can be reduced as follows.

\newpage

\begin{cor}\label{cor1.1}
Assume that $\Gamma$ is a nondegenerate bipartite graph. Then $\Gamma$ is $2$-lit. Moreover the following are equivalent:
\begin{enumerate}
\item $\Gamma$ is $1$-lit

\item  $\Gamma$ contains a vertex with even degree or $\Gamma$ is a single edge.
\end{enumerate}
\end{cor}

As consequences of Corollary~\ref{cor1.1}, we obtain two families of $1$-lit graphs as follows.

\begin{enumerate}
\item[$\bullet$] A tree is nondegenerate if and only if it has a perfect matching. By \cite[Lemma~2.4]{eigen_nondegtree:02}, a tree with a perfect matching satisfies Corollary~\ref{cor1.1}(ii) and is therefore $1$-lit (cf. \cite[Theorem~1.1]{hw:11-1}).
\item[$\bullet$] For any two positive integers $m$ and $n$, the $m\times n$ grid is nondegenerate if and only if $m+1$ and $n+1$ are coprime \cite{ks:90}. By Corollary~\ref{cor1.1} any such $m\times n$ gird is $1$-lit.
\end{enumerate}

The following example shows that Corollary~\ref{cor1.1} is no longer true if the assumption of $\Gamma$ is the same as that of Theorem~\ref{thm1.3}.  Consider the graph $\Gamma=(S,R)$ as below.

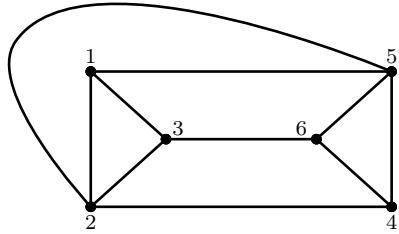
\psset{unit=2cm}
\begin{pspicture}(0,0)(8,1.9)
\psset{linewidth=1pt}
\psline{*-*}(3.00,0.3)(3.00,1.2)
\psline{*-*}(3.00,1.2)(5.00,1.2)
\psline{*-*}(3.00,1.2)(3.50,0.75)
\psline{*-*}(3.00,0.3)(3.50,0.75)
\psline{*-*}(5.00,1.2)(4.50,0.75)
\psline{*-*}(5.00,0.3)(4.50,0.75)
\psline{*-*}(5.00,0.3)(5.00,1.2)
\psline{*-*}(3.00,0.3)(5.00,0.3)
\psline{*-*}(3.50,0.75)(4.50,0.75)
\pscurve(5.00,1.2)(2.50,1.4)(3.00,0.3)

\rput(3.00,0.2){{\scriptsize $2$}}
\rput(3.00,1.3){{\scriptsize $1$}}
\rput(3.58,0.82){{\scriptsize $3$}}
\rput(4.40,0.82){{\scriptsize $6$}}
\rput(5.00,0.2){{\scriptsize $4$}}
\rput(5.00,1.3){{\scriptsize $5$}}
\end{pspicture}

\noindent The graph $\Gamma=(S,R)$ is nondegenerate and not a block graph. Therefore $\Gamma$ is not a line graph by Proposition~\ref{lem5.4}. The basis $\{\alpha_1^\vee,\alpha_2^\vee,\ldots,\alpha_6^\vee\}$ of $V$ dual to $\{\alpha_1,\alpha_2,\ldots,\alpha_6\}$ can be expressed as follows.
$$
\begin{array}{ll}
\alpha_1^\vee=\alpha_2+\alpha_6, \qquad \quad
&\alpha_4^\vee=\alpha_3+\alpha_5,\\
\alpha_2^\vee=\alpha_1+\alpha_3+\alpha_5+\alpha_6,\qquad \quad
&\alpha_5^\vee=\alpha_2+\alpha_3+\alpha_4+\alpha_6,\\
\alpha_3^\vee=\alpha_2+\alpha_4+\alpha_5,\qquad \quad
&\alpha_6^\vee=\alpha_1+\alpha_2+\alpha_5.
\end{array}
$$
Using (\ref{e2.2}) and $Q(\alpha_s)=1$ for all $s\in S$, we deduce that $Q(\alpha_s^\vee)=0$ for all $s\in S.$ Therefore $\Gamma$ is not $1$-lit by Theorem~\ref{thm1.3}, but the vertices $2,5$ have even degree in $\Gamma$.

\section{Proof of Theorem~\ref{thm1.2}}

For $\alpha\in V$ the {\it transvection on $V$ with direction $\alpha$} is a linear transformation  $\tau_\alpha:V\to V$ defined by
\begin{eqnarray*}
\tau_\alpha \beta = \beta+B(\beta,\alpha)\alpha \qquad \quad \hbox{for all $\beta\in V.$}
\end{eqnarray*}
Observe that $\tau_\alpha$ preserves the form $B$ and that $\tau_\alpha\in {\rm GL}(V)$ since $\tau_\alpha^2=1$. Here $1$ denotes the identity map on $V.$

For a subset $P$ of $V$ define $Tv(P)$ to be the subgroup of ${\rm GL}(V)$ generated by $\tau_\alpha$ for $\alpha\in P,$ and define $G(P)$ to be the simple graph whose vertex set is $P$ and where $\alpha,\beta$ in $P$ form an edge if and only if $B(\alpha,\beta)=1.$ For any two linearly independent sets $P$ and $P'$ of $V$, we say that $P'$ is {\it elementary $t$-equivalent} to $P$ whenever there exist $\alpha,\beta\in P$ such that $P'$ is obtained from $P$ by changing $\beta$ to $\tau_\alpha \beta.$ The equivalence relation generated by the elementary $t$-equivalence relation is called the {\it $t$-equivalence relation} \cite{brohum:86-1}.

\begin{lem}\label{lem2.3}
{\rm \cite[Theorem 3.3]{brohum:86-1}.}
Let $P$ denote a linearly independent set of $V$. Assume that $G(P)$ is a connected graph. Then there exists $P'$ in $t$-equivalence class of $P$ for which $G(P')$ is a tree.
\end{lem}

\begin{lem}\label{lem2.1}
{\rm \cite[Lemma 3.7]{hw:11-2}}.
Let $P$ denote a linearly independent set of $V$. Assume that $G(P)$ is the line graph of a tree. Then, for each $P'$ in the $t$-equivalence class of $P$,  the graph $G(P')$ is the line graph of a tree.
\end{lem}

A basis $P$ of $V$ is said to have {\it orthogonal type} \cite{brohum:86-2} if $P$ is $t$-equivalent to some $P'$ for which $G(P')$ is a tree containing the graph

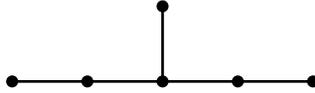
\psset{unit=2cm}
\begin{pspicture}(0,0)(8,1.1)
\psset{linewidth=1pt}

\psline(3.75,0.3)(3.75,0.8)
\psline(2.75,0.3)(4.75,0.3)
\pscircle*(2.75,0.3){0.04}
\pscircle*(3.25,0.3){0.04}
\pscircle*(3.75,0.3){0.04}
\pscircle*(4.25,0.3){0.04}
\pscircle*(4.75,0.3){0.04}
\pscircle*(3.75,0.8){0.04}
\end{pspicture}

\noindent as a subgraph.

\begin{lem}\label{lem2.5}
Assume that $P$ is a basis of $V$ for which $G(P)$ is a tree but not a path. Then $P$ is of orthogonal type.
\end{lem}
\begin{proof}
Since $G(P)$ is not a path it contains a vertex $\alpha$ with degree at least three. If any two neighbors of $\alpha,$ say $\beta$ and $\gamma$, are leaves of $G(P)$, then $\beta+\gamma$ lies in the radical of $V$, which contradicts that $B$ is nondegenerate. Therefore at most one neighbor of $\alpha$ is a leaf in $G(P)$ and so $P$ is of orthogonal type.
\end{proof}

\begin{lem}\label{lem2.2}
{\rm \cite[Section 10]{brohum:86-2}.}
Let $P$ denote a basis of $V$ which is of orthogonal type.
Then $Tv(P)$ is the orthogonal group with respect to $Q_P.$ Moreover the $Tv(P)$-orbits on $V$ are
\begin{gather*}
\{0\}, \qquad Q_P^{-1}(0)\setminus\{0\}, \qquad Q_P^{-1}(1).
\end{gather*}
\end{lem}



{\it Proof of Theorem~\ref{thm1.2}.}
For each $s\in S$ let $\tau_s$ denote the transvection on $V$ with direction $\alpha_s.$ By \cite[Section 5]{rd:05}, there exists a unique representation $\tau=\tau_{\it \Gamma}:W\rightarrow {\rm GL}(V)$ such that $\tau(s)=\tau_s$ for all $s\in S.$ For each $w\in W$ the transpose of $\tau(w^{-1})$ is equal to $\kappa(w)$. Therefore $\kappa$ is the dual representation of $\tau$. Since $\tau$ preserves the form $B$ we have
\begin{gather}\label{equi_rep}
\theta\circ\tau(w)=\kappa(w)\circ \theta \qquad  \qquad \hbox{for all $w\in W.$}
\end{gather}
Let $P=\{\alpha_s~|~s\in S\}$. Clearly $Tv(P)=\tau(W)$ and $G(P)$ is (isomorphic to) $\Gamma$. By Lemma~\ref{lem2.3} there exists $P'$ in $t$-equivalence class of $P$ for which $G(P')$ is a tree. Since $G(P)$ is not a line graph, the tree $G(P')$ is not a path by Lemma~\ref{lem2.1}. By Lemma~\ref{lem2.5} the basis $P'$ of $V$, as well as $P$, is of orthogonal type. By Lemma~\ref{lem2.2}, the group $\tau(W)=O(V)$ and the $\tau(W)$-orbits on $V$ are $\{0\},$  $Q^{-1}(0)\setminus\{0\},$ and $Q^{-1}(1).$ Applying (\ref{equi_rep}) and since $\theta$ is an isomorphism by Lemma~\ref{lem4.2}, the result follows.  \hfill $\Box$


\section{Proofs of Theorem~\ref{thm1.3} and its corollary}

To prove Theorem~\ref{thm1.3} and Corollary~\ref{cor1.1}, we introduce a simple graph which includes the information of the values $B(\alpha^\vee_s,\alpha^\vee_t)$ for all $s,t\in S.$

\begin{defn}\label{defn10.1}
We define $R^\vee$ as the set consisting of all two-element subsets $\{s,t\}$ of $S$ with $B(\alpha_s^\vee,\alpha_t^\vee)=1.$ Define $\Gamma^\vee$ as the simple graph with vertex set $S$  and edge set $R^\vee.$ We will refer to $\Gamma^\vee$ as the {\it dual graph} of $\Gamma.$ 
\end{defn}

Note that the notion of dual graphs defined above is different from the usual ones in graph theory. The following lemma suggests why the graph $\Gamma^\vee$ is of interest.


%

\begin{lem}\label{lem10.1}
For each $s\in S$ we have $\theta(\alpha_s^\vee)=f_s.$
\end{lem}
\begin{proof}
Let $s,t\in S$ be given. Using (\ref{e1.5}) and (\ref{e10.5}) we have $\theta(\alpha_s^\vee)\alpha_t=1$ whenever $s=t$ and otherwise $\theta(\alpha_s^\vee)\alpha_t=0.$ Comparing this with (\ref{e1.1}) the result follows.
\end{proof}

Recall from Section~\ref{Preliminaries} that the symplectic form $B$ is defined on the basis $\{\alpha_s~|~s\in S\}$ of $V$. If the symplectic form associated with $\Gamma^\vee$ is defined on the basis $\{\alpha_s^\vee~|~s\in S\}$ of $V$, then the resulting form is $B$. Therefore $\Gamma^\vee$ is a nondegenerate graph. The dual graph of $\Gamma^\vee$ is $\Gamma$ since $\{\alpha_s~|~s\in S\}$ is the basis of $V$ dual to $\{\alpha_s^\vee~|~s\in S\}.$

\begin{lem}\label{lem10.2}
For each $s\in S$ we have
\begin{eqnarray}\label{e10.1}
\alpha_s=\sum_{st\in R}\alpha_t^\vee.
\end{eqnarray}
\end{lem}
\begin{proof}
Fix $s\in S$. By (\ref{e1.1}), (\ref{e1.2}) and (\ref{e1.5}),  we find that the vector $\theta(\alpha_s)$ equals
\begin{eqnarray}\label{e10.6}
\sum_{st\in R}f_t.
\end{eqnarray}
By Lemma~\ref{lem10.1}, we find that (\ref{e10.6}) equals
$$
\theta\bigg(\sum_{st\in R}\alpha_t^\vee\bigg).
$$
Therefore (\ref{e10.1}) holds by Lemma~\ref{lem4.2}(ii).
\end{proof}

By duality and Lemma~\ref{lem10.2} the following lemma is straightforward.

\begin{lem}\label{lem10.4}
For each $s\in S$ we have
\begin{eqnarray}\label{e10.2}
\alpha_s^\vee=\sum_{st\in R^\vee}\alpha_t.
\end{eqnarray}
\end{lem}

\begin{lem}\label{lem10.3}
Let $A$ and $A^\vee$ denote the adjacency matrices of $\Gamma$ and $\Gamma^\vee$ over $\mathbb{F}_2,$ respectively. Then $A$ and $A^\vee$ are inverses of each other.
\end{lem}
\begin{proof}
We show that $A^\vee A$ is equal to the identity matrix. Let $s,t\in S$ be given. By the comment below Lemma~\ref{lem10.1} the $(s,t)$-entry of $A$ (resp. $A^\vee$) is equal to  $B(\alpha_s,\alpha_t)$ (resp. $B(\alpha_s^\vee,\alpha_t^\vee)$). By Definition~\ref{defn10.1} we find that the $(s,t)$-entry of $A^\vee A$ equals
\begin{align}\label{e10.7}
B\Bigg(\sum_{su\in R^\vee}\alpha_u,\alpha_t\Bigg).
\end{align}
By (\ref{e10.2}) the vector in the first coordinate of (\ref{e10.7}) equals $\alpha_s^\vee.$ Therefore (\ref{e10.7}) equals $1$ if and only if $s=t$ by (\ref{e10.5}). The result follows.
\end{proof}

We are now ready to prove Theorem~\ref{thm1.3}.

\medskip

{\it Proof of Theorem~\ref{thm1.3}.} In Lemma~\ref{lem10.1} we saw that $\theta(\alpha_s^\vee)=f_s$ for all $s\in S.$ Therefore (i), (ii) are equivalent by Theorem~\ref{thm1.2}. To show that $\Gamma$ is $2$-lit, it is now enough to consider the two cases: (a) $Q(\alpha_s^\vee)=0$ for all $s\in S;$ (b) $Q(\alpha_s^\vee)=1$ for all $s\in S.$

(a) It suffices to show that there exist $s,t\in S$ such that $Q(\alpha_s^\vee+\alpha_t^\vee)=1.$ Since the form $B$ is nontrivial there exist $s,t\in S$ such that $B(\alpha_s^\vee,\alpha_t^\vee)=1.$ Then the $s$ and $t$ are the desired elements in $S.$

(b) It suffices to show that there exist two distinct $s,t\in S$ such that $Q(\alpha_s^\vee+\alpha_t^\vee)=0.$ By our assumption, the graph $\Gamma$ is not a complete graph. Using Lemma~\ref{lem10.3}, we deduce that $\Gamma^\vee$ is not a complete graph. Therefore there exist two distinct $s,t\in S$ such that $B(\alpha_s^\vee,\alpha_t^\vee)=0.$ Such $s$ and $t$ are desired elements in $S.$ \hfill $\Box$


\medskip

To prove Corollary~\ref{cor1.1}, we give a sufficient condition for Theorem~\ref{thm1.3}(ii).

\begin{lem}\label{lem9.1}
Let $\Gamma=(S,R)$ denote a nondegenerate graph. Assume that there exists $s\in S$ with even degree in $\Gamma$ such that
\begin{gather}\label{e9.2}
\sum_{\scriptsize \{u,v\}\subseteq S
\atop \scriptsize su,sv\in R}
B(\alpha_u^\vee,\alpha_v^\vee)=0,
\end{gather}
where the sum is over all two-element subsets $\{u,v\}$ of $S$ with $su,sv\in R$. Then the restriction of $Q$ to $\{\alpha_t^\vee~|~st\in R\}$ is surjective.
\end{lem}
\begin{proof}
Apply $Q$ to either side of (\ref{e10.1}). Using (\ref{e2.2}), (\ref{e9.2}) and $Q(\alpha_s)=1$ to evaluate the resulting equation, we obtain that
\begin{eqnarray}\label{e9.1}
\sum_{st\in R}Q(\alpha_t^\vee)=1.
\end{eqnarray}
By (\ref{e9.1}) there exists a neighbor $u$ of $s$ for which $Q(\alpha_u^\vee)=1.$ Since $s$ has even degree in $\Gamma$ there exists a neighbor $v$ of $s$ for which $Q(\alpha_v^\vee)=0.$ The result follows.
\end{proof}


\smallskip

{\it Proof of Corollary~\ref{cor1.1}.}
By Proposition~\ref{lem5.4} a nondegenerate bipartite graph $\Gamma$ is a line graph if and only if $\Gamma$ is a path of even order. Since every path is $1$-lit, this corollary holds for $\Gamma$ as a line graph. We thus assume that $\Gamma$ is not a line graph. By Theorem~\ref{thm1.3} the graph $\Gamma$ is $2$-lit. By Lemma~\ref{lem10.3} we deduce that the graph $\Gamma^\vee$ is bipartite with bipartition as same as that of $\Gamma.$ We use this to show that (i), (ii) are equivalent.

(ii) $\Rightarrow$ (i): Let $s$ denote a vertex of $\Gamma$ with even degree. Since $\Gamma$ and $\Gamma^\vee$ are bipartite graphs with same bipartition, we deduce that $B(\alpha_u^\vee,\alpha_v^\vee)=0$ for any neighbors $u,v$ of $s$ in $\Gamma$. Therefore (\ref{e9.2}) holds. By Lemma~\ref{lem9.1} the restriction of $Q$ on $\{\alpha_t^\vee~|~st\in R\}$ is onto. Therefore $\Gamma$ is $1$-lit by Theorem~\ref{thm1.3}.

(i) $\Rightarrow$ (ii): Suppose on the contrary that each vertex of $\Gamma$ has odd degree. Using Lemma~\ref{lem10.3} we deduce that each vertex of $\Gamma^\vee$ has odd degree. Let $s$ denote any element of $S$. By (\ref{e10.2}) the value $Q(\alpha_s^\vee)$ equals
\begin{align}\label{e11.1}
Q\Bigg(\sum_{st\in R^\vee}\alpha_t \Bigg).
\end{align}
Since the bipartite graphs $\Gamma$ and $\Gamma^\vee$ have the same bipartition, we deduce that $B(\alpha_u,\alpha_v)=0$ for any neighbors $u,v$ of $s$ in $\Gamma^\vee.$ By (\ref{e2.2}) the summation in (\ref{e11.1}) can be moved out front. Since $Q(\alpha_s)=1$ for all $s\in S$, it follows that (\ref{e11.1}) equals $1$, contradicting to Theorem~\ref{thm1.3}(ii).
\hfill $\Box$

\medskip








\medskip

\noindent Hau-wen Huang
\hfil\break Mathematics Division
\hfil\break National Center for Theoretical Sciences
\hfil\break National Tsing-Hua University
\hfil\break Hsinchu 30013, Taiwan, R.O.C.
\hfil\break Email:  {\tt hauwenh@math.cts.nthu.edu.tw}
\medskip

\end{document}